\documentclass[11pt,leqno]{article}
\usepackage[utf8]{inputenc}
\usepackage{amsmath}
\usepackage{amsfonts}
\usepackage{amssymb}
\usepackage{graphicx}
\usepackage{mathrsfs}

\usepackage{mathrsfs}\usepackage{upref,amsthm,amsxtra,exscale}
\usepackage{cite}
\usepackage[colorlinks=true,urlcolor=blue,
citecolor=red,linkcolor=blue,linktocpage,pdfpagelabels,
bookmarksnumbered,bookmarksopen]{hyperref}
\allowdisplaybreaks[4]

\newtheorem{theorem}{Theorem}[section]
\newtheorem{corollary}[theorem]{Corollary}
\newtheorem{remark}[theorem]{Remark}
\newtheorem{problem}[theorem]{Problem}
\newtheorem{lemma}[theorem]{Lemma}
\newtheorem{proposition}[theorem]{Proposition}

\numberwithin{equation}{section}

\def\r{\mathbb{R}}
\def\rn{\mathbb{R}^N}
\def\hn{\mathbb{H}^N}

\def\n{\mathbb{N}}

\def\eps{\varepsilon}
\def\vr{\varrho}
\def\rh{\rightharpoonup}
\def\io{\int_{\Omega}}

\def\wt{\widetilde}
\def\wh{\widehat}
\def\ol{\overline}
\def\cC{\mathcal{C}}

\def\cH{\mathcal{H}}

\def\cJ{\mathcal{J}}

\def\cN{\mathcal{N}}
\def\cO{\mathcal{O}}
\def\cS{\mathscr{S}}
\def\cT{\mathcal{T}}
\def\cU{\mathcal{U}}
\def\cV{\mathcal{V}}

\author{Mónica Clapp\footnote{M. Clapp was partially supported by CONACYT grant A1-S-10457 (Mexico).}\; and Andrzej Szulkin}
\title{Non-variational weakly coupled elliptic systems}
\date{}

\begin{document}

\maketitle

\begin{abstract}
We establish the existence of a nonnegative fully nontrivial solution to a non-variational weakly coupled competitive elliptic system. We show that this kind of solutions belong to a topological manifold of Nehari-type, and apply a degree-theoretical argument on this manifold to derive existence.

\medskip


\medskip

\end{abstract}

\section{Introduction and statement of results} \label{sec:intro}

In this paper we consider the existence of solutions to the elliptic system
\begin{equation} \label{eq:system}
\begin{cases}
-\Delta u_i = \mu_i u_i^p + \sum\limits_{j\neq i}\lambda_{ij}u_i^{\alpha_{ij}}u_j^{\beta_{ij}}, \\
u_i\ge 0,\ u_i\not\equiv 0 \text{ in } \Omega, \\
u_i\in H^1_0(\Omega), \quad i,j=1,\ldots,\ell,
\end{cases}
\end{equation}
where $\Omega$ is a smooth bounded domain in $\rn$, $N\ge 2$, $1<p<\frac{N+2}{N-2}$ if $N\geq 3$, $1<p<\infty$ if $N=2$, $\mu_i>0$, $\lambda_{ij}<0$, $\alpha_{ij},\beta_{ij}>0$ and $\alpha_{ij}+\beta_{ij}<p$ for $i,j=1,\ldots,\ell$, $j\ne i$. This system arises as a model for the steady state distribution of $l$ competing species coexisting in $\Omega$. Here $u_i$ represents the density of the $i$-th population, $\mu_i$ corresponds to the attraction between the species of the same kind, or more generally, $\mu_iu_i^p$ can be replaced by $f_i(u_i)$ and represent internal forces. The parameters $\lambda_{ij}$, $\lambda_{ji}$ (which may not be equal) correspond to the interaction (repulsion) between different species. In particular, if $\alpha_{ij}=\beta_{ij}=1$, then the interaction is of the Lotka-Volterra type while $\alpha_{ij}=1$, $\beta_{ij}=2$ corresponds to the interaction which appears in the Bose-Einstein condensates. In the latter case one also has $\lambda_{ij}=\lambda_{ji}$ and the system is variational. 

In what follows we do not assume $\lambda_{ij}=\lambda_{ji}$ or $\beta_{ij}=\alpha_{ji}$. The system \eqref{eq:system} is non-variational except for some very special choices of $\lambda_{ij}$, $\alpha_{ij}$ and $\beta_{ij}$.  While there is an extensive literature concerning the existence (and multiplicity) of solutions for variational systems like \eqref{eq:system}, there are not so many results in the non-variational case. Here we could mention \cite{cpq, cd, ctv, dd, dd2} where, however, the right-hand sides are quite different from ours. In particular, in \cite{cd, ctv, dd, dd2} the interaction term is of the Lotka-Volterra type (or is a variant of it) while the terms $f_i(u_i)$ are different from $\mu_iu_i^p$. For these $f_i$ one obtains uniform bounds on the solutions when $\lambda_{ij}\to-\infty$. Existence of such bounds allows to study the limiting behaviour of solutions. To be more precise, if $\lambda_{ij,n}\to-\infty$ and $(u_{1,n},\ldots,u_{l,n})$ is a corresponding solution with uniform bound on each component, then one expects that $u_{i,n}\to u_i$ (in an appropriate space) and $u_i(x)\cdot u_j(x)=0$ a.e.\ in $\Omega$ for all $i\ne j$, i.e. different components separate spatially. This has been   studied in the above mentioned papers. In \cite{cpq, ctv} the emphasis is in fact on the properties of limiting configurations, including regularity of free boundaries between the components. 

The main result of this paper is the following

\begin{theorem} \label{mainthm}
The system \eqref{eq:system} has a solution.
\end{theorem}

Existence proofs in the above-mentioned papers do not seem to be applicable here. Our problem can be reformulated as an operator equation in the space $\cH := H^1_0(\Omega)^\ell$ and one can use degree theory  to obtain a nontrivial solution. However, this could give a semitrivial solution (i.e. $u_i=0$ for some but not all $i$). To rule out such solutions we introduce a Nehari-type manifold on which all  $u$ are fully nontrivial in the sense that no $u_i$ is identically zero,  and then we apply a degree-theoretical argument on this manifold. 

 We do not know if there always exist solutions for \eqref{eq:system} which are uniformly bounded, see Problem \ref{pb}. Moreover, as we shall see in Section \ref{sec:sync}, under a suitable choice of exponents and parameters and for $\ell=2$ there exists a sequence of solutions which are synchronized in the sense that $u_{i,n}=t_{i,n}v_n$ ($i=1,2$) and such that $\|u_{i,n}\|\to\infty$ as $\lambda_{12,n},\lambda_{21,n}\to-\infty$. So the components neither separate spatially nor are bounded.

\medskip

Let $u_i^+ := \max\{u_i,0\}$, $u_i^- := \min\{u_i,0\}$, and consider the system
\begin{equation} \label{eq:system2}
\begin{cases}
-\Delta u_i = \mu_i (u_i^+)^p + \sum\limits_{j\neq i}\lambda_{ij}(u_i^+)^{\alpha_{ij}}(u_j^+)^{\beta_{ij}}, \\
u_i\in H^1_0(\Omega), \quad i,j=1,\ldots,\ell.
\end{cases}
\end{equation}
In Proposition \ref{mainprop}(v) we shall show that any fully nontrivial solution to this system also solves \eqref{eq:system}.

 In what follows we shall work with \eqref{eq:system2} and we shall also need the parametrized system
\begin{equation} \label{eq:system3}
\begin{cases}
-\Delta u_i = \mu_i (u_i^+)^p + t\sum\limits_{j\neq i}\lambda_{ij}(u_i^+)^{\alpha_{ij}}(u_j^+)^{\beta_{ij}}, \\
u_i\in H^1_0(\Omega), \quad i,j=1,\ldots,\ell, \quad 0\le t\le 1.
\end{cases}
\end{equation}
Note that \eqref{eq:system3} homotopies \eqref{eq:system2} to an uncoupled system. Since
\[
t\sum\limits_{j\neq i}|\lambda_{ij}|(u_i^+)^{\alpha_{ij}}(u_j^+)^{\beta_{ij}} \le C(1+(u_1^+)^q+\cdots + (u_\ell^+)^q),
\]
where $\alpha_{ij}+\beta_{ij}\le q<p$ for all $i,j$, the following statement holds true.

\begin{lemma} \label{Linfty}
All solutions $u =(u_1,\ldots,u_\ell)$ of \eqref{eq:system3} are uniformly bounded in $L^\infty(\Omega)$ and hence in $H^1_0(\Omega)$. This bound is independent of $t\in[0,1]$.
\end{lemma}

This has been shown, in a much more general setting, in \cite{gs} for a single equation and in \cite{dFY} for two equations. It is easy to see that the argument in \cite{dFY} extends to an arbitrary number of equations. In both papers a blow-up procedure is used in order to reduce the problem to a Liouville-type result. For the reader's convenience, in Appendix \ref{sec:appendix} we shall provide a simple proof of such reduction, adapted to our special case. The assumption $q<p$ is crucial for the validity of this lemma. Indeed, in \cite{dww} it has been shown that the conclusion may fail if $q=p$.

The paper is organized as follows. In Section \ref{sec:rm} we state and prove a lemma for functions in $\r^\ell$. In Section \ref{sec:nehari} we define a Nehari-type manifold $\cN$ similar to the one introduced in \cite{ctv2}.  We also show that solutions to \eqref{eq:system2} correspond to solutions for an operator equation in an open subset of the product of the unit spheres $\cS_i\subset H^1_0(\Omega)$, $1\le i\le \ell$. The idea comes from \cite{cs}. To our knowledge, this is the first time a Nehari-type manifold appears in a non-variational setting.  Theorem \ref{mainthm} is proved in Section \ref{sec:main} and synchronized solutions are discussed in Section \ref{sec:sync}. As we have already mentioned, Lemma \ref{Linfty} is proved in Appendix \ref{sec:appendix}.

\section{A lemma on functions in $\r^\ell$} \label{sec:rm}

Let $a_i,\alpha_{ij},\beta_{ij}>0$, $b_i,d_{ij}\ge 0$, $\alpha_{ij}+\beta_{ij}<p$ for all $i,j=1,\ldots,\ell$, $j\ne i$. Define $M:(0,\infty)^\ell\to\r^\ell$ as
\[M(s) := (M_1(s),\ldots,M_\ell(s)),
\]
where
\[
M_i(s) := a_is_i - b_is_i^p + \sum\limits_{j\neq i}d_{ij}s_i^{\alpha_{ij}}s_j^{\beta_{ij}}, \qquad i,j= 1,\ldots,\ell.
\]

\begin{lemma} \label{rm}
\begin{itemize}
\item[$(i)$] If $b_i=0$ for some $i$, then $M(s)\ne 0$ for any $s\in(0,\infty)^\ell$.
\item[$(ii)$] If $b_i>0$ for all $i$, then there exists $s\in(0,\infty)^\ell$ such that $M(s)=0$.

 Moreover, if \ $0<a\le a_i\le\ol a$, \ $0<b\le b_i\le \ol b$ \ and \ $d_{ij}\le\ol d$ \ for all $i,j$, then there exist $0<r<R$, depending only on $a,\ol a,b,\ol b,\ol d$, such that $s\in(r,R)^\ell$.
\item[$(iii)$] The solution $s$ in $(ii)$ is unique.
\item[$(iv)$] The solution $s$ in $(ii)$ depends continuously on $a_i,b_i>0,\ d_{ij}\geq 0$.
\end{itemize}
\end{lemma}

\begin{proof}
$(i):$ If $b_i=0$ then 
\[
M_i(s) = a_is_i+\sum\limits_{j\neq i}d_{ij}s_i^{\alpha_{ij}}s_j^{\beta_{ij}}>0 \qquad \text{for all } s\in(0,\infty)^\ell.
\]

$(ii):$  Let $0<r<R$ be such that, for every $i,j=1,\ldots,\ell$,
\begin{align*}
a_it-b_it^p >0 &\qquad\text{if \ }t\in(0,r], \\
a_it-b_it^p+\sum\limits_{j\neq i}d_{ij}t^{\alpha_{ij}+\beta_{ij}}<0 &\qquad\text{if \ }t\in[R,\infty)
\end{align*}
(such $R$ exists because $\alpha_{ij}+\beta_{ij}<p$).
If $s=(s_1,\ldots,s_\ell)\in(0,\infty)^\ell$ and $s_i\ge s_j$ for all $j$, then
\[
M_i(s) = a_is_i-b_is_i^p+\sum\limits_{j\neq i}d_{ij}s_i^{\alpha_{ij}}s_j^{\beta_{ij}} \le  a_is_i-b_is_i^p+\sum\limits_{j\neq i}d_{ij}s_i^{\alpha_{ij}+\beta_{ij}}.
\]
Therefore, $M_i(s)<0$ whenever $s_i=\max\{s_1,\ldots,s_\ell\}\ge R$, and $M_i(s)>0$ if $0<s_i\le r$. If \ $a\le a_i\le\ol a$, \ $b\le b_i\le \ol b$, \ $d_{ij}\le\ol d$, \ then 
\[
a_it-b_it^p\ge at-\ol bt^p,  \qquad  a_it-b_it^p+\sum\limits_{j\neq i}d_{ij}t^{\alpha_{ij}+\beta_{ij}} \le \ol at-bt^p + \sum\limits_{j\neq i}\ol dt^{\alpha_{ij}+\beta_{ij}},
\]
so $r,R$ may be chosen as claimed.

Let
\[
G(s) := \rho-s \qquad \text{where}\quad\rho :=\tfrac{r+R}2(1,\ldots,1).
\] 
Then  $H(s,\tau) := \tau M(s)+(1-\tau)G(s) \ne 0$ on the boundary of $[r,R]^\ell$ for every $\tau\in[0,1]$. Hence this is an admissible homotopy for the Brouwer degree (see e.g. \cite[Appendix D]{wi} for the definition and properties of this degree). So 
\[
\deg(M, (r,R)^\ell, \rho) = \deg(G,(r,R)^\ell, \rho) = (-1)^\ell
\]
and $M(s)=0$ must have a solution.

$(iii):$ If $M(s_1^0,\ldots,s_\ell^0)=0$, then $\wt M(1,\ldots,1)=0$ where
\[
\wt M_i(s) = \wt a_is_i-\wt b_1s_i^p+\sum\limits_{j\neq i}\wt d_{ij}s_i^{\alpha_{ij}}s_j^{\beta_{ij}}
\]
with $\wt a_i := a_is_i^0$, \ $\wt b_i := b_i(s_i^0)^p$, \ $\wt d_{ij} := d_{ij}(s_i^0)^{\alpha_{ij}}(s_j^0)^{\beta_{ij}}$.
So we may assume without loss of generality that $M(1,\ldots,1)=0$. Then, 
\[
a_i-b_i+\sum\limits_{j\neq i}d_{ij}=0.
\]
Suppose there is another solution $s=(s_1,\ldots,s_\ell)$. Then, using the previous identity, we get
\[
0=a_is_i-b_is_i^p+\sum\limits_{j\neq i}d_{ij}s_i^{\alpha_{ij}}s_j^{\beta_{ij}} = a_is_i - \Big(a_i+\sum\limits_{j\neq i}d_{ij}\Big)s_i^p+\sum\limits_{j\neq i}d_{ij}s_i^{\alpha_{ij}}s_j^{\beta_{ij}}, 
\]
and after rearranging the terms,
\[
a_i(s_i-s_i^p) = \sum\limits_{j\neq i}d_{ij}(s_i^p-s_i^{\alpha_{ij}}s_j^{\beta_{ij}}).
\]
There are two possible cases: If $s_i>1$ for some $i$, we may assume without loss of generality that $s_i\ge s_j$ for all $j$. Then the left-hand side above is negative while the right-hand side is $\ge 0$, a contradiction. If, on the other hand, $0<s_i<1$ for some $i$, we may assume $s_i\le s_j$ for all $j$. Now the left-hand side is positive and the right-hand side is $\le 0$, a contradiction again. 

$(iv):$ If $a_{n,i},a_i,b_{n,i},b_i>0$, $d_{n,i},d_i\geq 0$, $a_{n,i}\to a_i$, $b_{n,i}\to b_i$, $d_{n,ij}\to d_{ij}$ then, as in $(ii)$, there exist $0<r<R$ such that the unique solution $s_n$ to 
$$M_{n,i}(s) := a_{n,i}s_i - b_{n,i}s_i^p + \sum\limits_{j\neq i}d_{n,ij}s_i^{\alpha_{ij}}s_j^{\beta_{ij}}=0, \qquad i,j= 1,\ldots,\ell,$$
belongs to $[r,R]^\ell$ for every $n$. Passing to a subsequence, we have that $s_n\to s\in[r,R]^\ell$ and $M(s)=0$.
\end{proof}

\section{A Nehari-type manifold} \label{sec:nehari}

Let $\cH := H^1_0(\Omega)^\ell$, $u=(u_1,\ldots,u_\ell)\in\cH$. As convenient norms in $H^1_0(\Omega)$ and  $\cH$ we choose
\[
\|u_i\| :=\left(\io|\nabla u_i|^2\right)^{\frac12} \quad \text{and} \quad \|u\| := (\|u_1\|^2+\cdots+\|u_\ell\|^2)^{\frac12},
\]
and we denote by $\langle\,\cdot\,,\,\cdot\,\rangle$ the inner product in $H^1_0(\Omega)$. Let 
\begin{equation*} \label{eq:n}
I(u) := (I_1(u),\ldots,I_\ell(u))
\end{equation*}
where $I_i: H^1_0(\Omega)\to H^1_0(\Omega)$ are given by 
\begin{equation} \label{eq:ii}
I_i(u):=u_i-K_i(u)
\end{equation}
and
\begin{equation} \label{eq:ki}
\langle K_i(u),v\rangle := \io\mu_i(u_i^+)^pv + \sum\limits_{j\neq i}\lambda_{ij}\io(u_i^+)^{\alpha_{ij}}(u_j^+)^{\beta_{ij}}v\quad \forall v\in H^1_0(\Omega).
\end{equation}

\begin{lemma} \label{lem:K compact}
If $u_n\rh u$ weakly in $\cH$, then $K_i(u_n)\to K_i(u)$ strongly in $H^1_0(\Omega)$ for each $i=1,\ldots,\ell$. 
\end{lemma}

\begin{proof}
Since $p,\,\alpha_{ij}+\beta_{ij}<\frac{N+2}{N-2}$ for $N\ge3$, after passing to a subsequence $u_{n,i}^+\to u_i^+$ strongly in $L^{p+1}(\Omega)$ and in $L^{\alpha_{ij}+\beta_{ij}+1}(\Omega)$ for every $j\neq i$. Using Hölder's and Sobolev's inequalities we obtain
\begin{align*}
&|\langle K_i(u_n)-K_i(u),v\rangle| \leq C\Big(\big|(u_{n,i}^+)^p-(u_i^+)^p\big|_{\frac{p+1}{p}}\\
&+ \sum_{j\neq i}\big|(u_{n,i}^+)^{\alpha_{ij}}-(u_i^+)^{\alpha_{ij}}\big|_{\frac{\alpha_{ij}+\beta_{ij}+1}{\alpha_{ij}}}+ \sum_{j\neq i}\big|(u_{n,j}^+)^{\beta_{ij}}-(u_j^+)^{\beta_{ij}}\big|_{\frac{\alpha_{ij}+\beta_{ij}+1}{\beta_{ij}}}\Big)\|v\|,
\end{align*}
where $|\,\cdot\,|_r$ denotes the norm in $L^r(\Omega)$. From \cite[Theorem A.2]{wi} we derive
$$\sup_{v\neq 0}\frac{|\langle K_i(u_n)-K_i(u),v\rangle|}{\|v\|}\longrightarrow 0.$$
Hence, $K_i(u_n)\to K_i(u)$ strongly in $H^1_0(\Omega)$, as claimed.
\end{proof}

We define a Nehari-type set $\cN$ by putting 
\[
\cN := \{u\in\cH: u_i\ne 0 \text{ \ and \ } \langle I_i(u),u_i\rangle = 0 \text{ \ for all \ } i=1,\ldots,\ell\}.
\]

\begin{lemma} \label{bddaway}
$\cN$ is closed in $\cH$. 
\end{lemma}

\begin{proof}
Since $\lambda_{ij}<0$, it follows from the Sobolev inequality that
\[
\|u_i\|^2 \le \mu_i\io (u_i^+)^{p+1} \le C_i\|u_i\|^{p+1}
\]
for some $C_i>0$. Hence there exists $d_0>0$ such that, if $(u_1,\ldots,u_\ell)\in\cN$, then $\|u_i\| \ge d_0$ for all $i$. This shows that $\cN$ is closed in $\cH$.
\end{proof}

For $u :=(u_1,\ldots,u_\ell)\in\cH$, $s:=(s_1,\ldots,s_\ell)\in(0,\infty)^\ell$ and $su := (s_1u_1,\ldots,s_\ell u_\ell)$, we define
$$M_u(s) := (M_{u,1}(s),\ldots,M_{u,\ell}(s)),$$
where
\[
M_{u,i}(s) := \langle I_i(su),u_i\rangle = a_{u,i}s_i-b_{u,i}s_i^p + \sum\limits_{j\neq i}d_{u,ij}s_i^{\alpha_{ij}}s_j^{\beta_{ij}}
\] 
and
\[
a_{u,i} := \|u_i\|^2, \quad b_{u,i} := \io\mu_i(u_i^+)^{p+1}, \quad d_{u,ij} := \io(-\lambda_{ij})(u_i^+)^{\alpha_{ij}+1}(u_j^+)^{\beta_{ij}}.
\]

\begin{lemma} \label{su}
\begin{itemize}
\item[$(i)$]If $a_{u,i}\neq 0$ and $b_{u,i}=0$ for some $i$, then $M_{u}(s)\neq 0$ for any $s\in(0,\infty)^\ell$.
\item[$(ii)$]If $a_{u,i},\,b_{u,i}>0$ for all $i$, then there exists a unique $s_u\in(0,\infty)^\ell$ such that $M_u(s_u)= 0$. 
Moreover, if \ $0<a\le a_{u,i}\le\ol a$, \ $0<b\le b_{u,i}\le \ol b$ \ and \ $d_{u,ij}\le\ol d$ for all $i,j$, then there exist $0<r<R$, depending only on $a,\ol a,b,\ol b,\ol d$, such that $s_u\in(r,R)^\ell$.
\end{itemize}
\end{lemma}

\begin{proof}
This is an immediate consequence of Lemma \ref{rm}.
\end{proof}

Let 
\[
\cS := \{v\in H^1_0(\Omega): \|v\| = 1\}, \quad  \cT := \cS^\ell,
\]
and

\begin{align} 
\label{defU}  \cU :&= \{u\in\cT : s_u\in(0,\infty)^\ell \text{ exists with }M_u(s_u)=0\}\\
& \nonumber = \{u\in\cT: u_i^+\ne 0\text{ for all } i=1,\ldots,\ell\}.
\end{align}

The tangent space of $\cT$ at $u$ is 
\begin{equation} \label{tangent}
T_u(\cT):=\{(v_1,\ldots,v_\ell)\in\cH:\langle u_i,v_i\rangle=0\text{ for all }i=1,\ldots,\ell\}.
\end{equation}

\begin{proposition} \label{mainprop}
\begin{itemize}
\item[$(i)$] $\cU$ is a nonempty open subset of $\cT$ and \ $\cU\ne \cT$. 
\item[$(ii)$] The mapping $m: \cU\to\cN$ given by $m(u) := s_uu$ is a homeomorphism. In particular, $\cN$ is a topological manifold.
\item[$(iii)$] If $(u_n)$ is a sequence in $\cU$ such that $u_n\to u\in\partial\cU$, then $s_{u_n}\to\infty$ (and hence $\|m(u_n)\|\to\infty$). 
\item[$(iv)$] Let $S:\cU\to\cH$ be given by
\[
S(u) := I(s_uu) = s_uu-K(s_uu).
\]
Then $S(u)\in T_u(\cU)$ for every $u\in\cU$. 
\item[$(v)$] $S(u) = 0$ if and only if $m(u)=s_uu$ is a solution for \eqref{eq:system}.
\end{itemize}
\end{proposition}

\begin{proof}
$(i):$ That $\cU$ is neither empty nor the whole $\cT$ is obvious and, since $u\mapsto u_i^+$ is continuous \cite[Lemma 2.3]{ccn}, it is easily seen from the second line of \eqref{defU} that $\cU$ is open in $\cT$.

$(ii):$ If $u\in\cU$, then $s_uu\in\cN$ because $\langle I_i(s_uu),s_{u,i}u_i\rangle = s_{u,i}M_{u,i}(s_u)= 0$ for all $i$. So $m$ is well defined.  If $(u_n)$ is a sequence in $\cU$ and $u_n\to u\in \cU$, then  $a_{u_n,i}\to a_{u,i}$, $b_{u_n,i}\to b_{u,i}$ and $d_{u_n,ij}\to d_{u,ij}$ for all $i,j$. By Lemma \ref{rm}$(iv)$, $s_{u_n}\to s_u$. Hence, $m$ is continuous. 

If $u\in\cN$, then $u_i^+\neq 0$ for all $i$. Otherwise, $0=\langle I_i(u),u_i\rangle=\|u_i\|^2$, a contradiction. Hence, the inverse of $m$ satisfies
\[
m^{-1}(u) := \left(\frac{u_1}{\|u_1\|}, \ldots, \frac{u_\ell }{\|u_\ell\|}\right)\in\cU,
\]
and it is obviously continuous.

$(iii):$ Let $(u_n)$ be a sequence in $\cU$ such that $u_n\to u\in \partial\cU$. If $(s_{u_n})$ is bounded, then, after passing to a subsequence, $s_{u_n}\to s_*$. Since $\cN$ is closed, $s_*u\in\cN$ and hence $u\in\cU$. This is impossible because $\cU$ is open.

$(iv):$ Since $\langle I_i(s_uu),u_i\rangle = M_{u,i}(s_u)= 0$ for all $i$, we have that $S(u)\in T_u(\cT)$ according to \eqref{tangent}.

$(v):$ If $u\in\cU$ satisfies $S(u)=0$, then $\bar u:=s_uu\in\cN$ and $\bar u$ is a weak solution to the system \eqref{eq:system2} (see \eqref{eq:ii} and \eqref{eq:ki}). Multiplying the $i$-th equation in \eqref{eq:system2} by $u_i^-:=\min\{\bar u_i,0\}$ and integrating gives $\io|\nabla u_i^-|^2 = 0$. 
Hence $u_i^- = 0$, i.e., $\bar u_i\geq 0$ for all $i$. As $\bar u\in\cN$, we have that $\bar u_i\neq 0$. This proves that $\bar u$ solves \eqref{eq:system}. The converse is obvious.
\end{proof}

\begin{remark} \label{rem:pos}
\emph{ 
If $\alpha_{ij}\ge 1$ for all $i$ and all $j\ne i$, then, as $\bar u_i$ above satisfies the $i$-th equation in \eqref{eq:system}, we have
\[
-\Delta \bar u_i +c(x)\bar u_i\ge 0 \quad \text{where } c(x) := -\sum\limits_{j\neq i}\lambda_{ij}\bar u_i^{\alpha_{ij}-1}\bar u_j^{\beta_{ij}}.
\]
Since all $u_i$ are continuous in $\overline\Omega$ and $c\ge 0$, it follows from the strong maximum principle (see e.g. \cite[Theorem 3.5]{gt}) that our solution is strictly posi\-tive in $\Omega$ in this case.
}
\end{remark}

\section{Proof of Theorem \ref{mainthm}} \label{sec:main}

In this section the sub- or superscript $t$ will be used in order to emphasize that we are concerned with the system \eqref{eq:system3}. So, e.g., 
\begin{equation} \label{eq:it} I_{t}(u):=u-K_{t}(u),\qquad S_t(u):=I_{t}(s^t_uu)=s^t_uu-K_t(s^t_uu),
\end{equation}
with
\begin{equation*}
\langle K_{t,i}(u),v\rangle := \io\mu_i(u_i^+)^pv + t\sum\limits_{j\neq i}\lambda_{ij}\io(u_i^+)^{\alpha_{ij}}(u_j^+)^{\beta_{ij}}v,
\end{equation*}
and
\[
\cN_t := \{u\in\cH: u_i\ne 0,\  \langle I_{t,i}(u),u_i\rangle = 0 \text{ for all } i=1,\ldots,\ell\}.
\] 
According to this notation, $\cN_1 =\cN$. When $t=1$, we shall sometimes omit the sub- or superscript $t$.

Consider first the system \eqref{eq:system3} with $t=0$. In this case the equations are uncoupled, the set
$$\cN_0=\{u\in\cH:u_i\ne 0,\ \|u_i\|^2=\io\mu_i(u_i^+)^{p+1}\text{ \ for all \ }i=1,\ldots,\ell\}$$
is the product of the usual Nehari manifolds associated to these equations, and the components of $s_u^0=(s_{u,1}^0\ldots,s_{u,\ell}^0)$ are
$$s_{u,i}^0=\left(\io\mu_i(u_i^+)^{p+1}\right)^{-\frac{1}{p-1}}, \qquad  u\in\cU.$$
The function $I_0$ (cf. \eqref{eq:it}) is the gradient vector field of the functional $\cJ:\cH\to\r$ defined by
\[
\cJ(u):= \frac12\sum_{i=1}^\ell\|u_i\|^2 - \frac1{p+1}\sum_{i=1}^\ell\io \mu_i(u_i^+)^{p+1}.
\]
Note that 
\begin{equation*} 
\cJ(u)=(\tfrac{1}{2}-\tfrac{1}{p+1})\|u\|^2\quad\text{if \ }u\in\cN_0.
\end{equation*}
$\cJ$ has a minimizer $\wt u_0=(\wt u_{0,1},\ldots,\wt u_{0,\ell})$ on $\cN_0$ with $\wt u_{0,i}>0$ and $\wt u_0$ is a solution to the system \eqref{eq:system3} with $t=0$. Each component $\wt u_{0,i}$ is a positive least energy solution to the $i$-th equation of this system. Let $\Psi:\cU\to\r$ be given by
\begin{align} \label{eq:psi}
\Psi(u) :&=\cJ(s_u^0u) =\left(\tfrac{1}{2}-\tfrac{1}{p+1}\right)|s_u^0|^2\\
 &=\left(\tfrac{1}{2}-\tfrac{1}{p+1}\right)\sum_{i=1}^\ell\left(\io\mu_i(u_i^+)^{p+1}\right)^{-\frac{2}{p-1}}. \nonumber
\end{align}
By \cite[Proposition 1.12]{wi} one has that $\Psi\in\cC^2(\cU,\r)$. It is easily seen that
\begin{equation} \label{eq:psi'}
\Psi'(u)v=\cJ'(s^0_uu)[s^0_uv]=\langle S_0(u),s^0_uv\rangle\quad\text{for all \ }u\in\cU, \ v\in T_u(\cU),
\end{equation}
and that $u$ is a critical point of $\Psi$ if and only if $u\in\cU$ and $m_0(u)=s^0_uu$ is a critical point of $\cJ$, see \cite[Theorem 3.3]{cs}. Let $u_0:=m_0^{-1}(\wt u_0)$. Then $u_0$ is a minimizer for $\Psi$.

Invoking Lemma \ref{Linfty} we may choose $R>0$ such that all solutions to the systems \eqref{eq:system3} are contained in the open ball $B_R(0)\subset\cH$, where $R$ is independent of $t\in[0,1]$. Then, by Proposition  \ref{mainprop}, 
\begin{equation} \label{cv}
\{u\in\cU:S_t(u)=0\}\subset\cV_t:=m_t^{-1}(B_R(0)\cap \cN_t).
\end{equation}
For $a\le d$ let
$$\Psi^d:=\{u\in\cU:\Psi(u)\leq d\}, \quad \Psi_a^d := \{u\in\cU: a\le \Psi(u)\leq d\}.$$
It follows from Proposition \ref{mainprop}$(iii)$ that the set $\Psi^d$
is closed in $\cT$ for any $d\in\r$. Note that $\lambda_{ij}<0$ implies $s_{u,i}^t\geq s_{u,i}^0$ for every $u\in\cU$, $t\in[0,1]$, $i=1,\ldots\ell$. So if $|s_u^t|<R$, then $|s_u^0|<R$; hence $\cV_t\subset\cV_0$ and, setting $c:=(\frac{1}{2}-\frac{1}{p+1})R^2$, we have that the closure of $\cV_t$ in $\cT$ satisfies
\begin{equation} \label{vt}
\overline{\cV}_t\subset\overline{\cV}_0\subset\Psi^c\qquad\forall t\in[0,1].
\end{equation}

For each $i=1,\ldots,\ell$ and $k\geq 2$ we choose an ascending sequence $(E_{k,i})$ of linear subspaces of $H^1_0(\Omega)$ such that \ $\dim E_{k,i} = k$, $u_{0,i}\in E_{2,i}$ ($u_0$ is the minimizer chosen above) and $\overline{\bigcup_{k\ge1} E_{k,i}} = H^1_0(\Omega)$. We define 
\[
E_k:=E_{k,1}\times\cdots\times E_{k,\ell}\subset\cH
\]
and denote by $P_k$ the orthogonal projector of $\cH$ onto $E_k$.

\begin{lemma} \label{lem:critinV}
Given $d>0$ there exists $k_d\in\n$ such that
$$P_k(S_t(u))\neq 0\qquad\text{for all \ }u\in(\Psi^d\smallsetminus\cV_t)\cap E_k, \ k\geq k_d, \ t\in[0,1].$$
\end{lemma}

\begin{proof}
Arguing by contradiction, assume that there exist $k_n\to\infty$, $t_n\in[0,1]$ and $u_n\in(\Psi^d\smallsetminus\cV_{t_n})\cap E_{k_n}$ such that
\begin{equation} \label{eq:critinV}
P_{k_n}(S_{t_n}(u_n))=s_{u_n}^{t_n}u_n-P_{k_n}K_{t_n}(s_{u_n}^{t_n}u_n)=0\qquad\forall n\in\n.
\end{equation} 
As $u_n\in\Psi^d$, we derive from \eqref{eq:psi} that $\io\mu_i(u_{n,i}^+)^{p+1}\ge b$ for some $b>0$ and all $n, i$. In the notation of Lemma \ref{su}, we have $a_{u_n,i}=1$ and, using the Hölder and the Sobolev inequalities, $b\le b_{u_n,i}\le \ol b$ and 
\[
d_{u_n,ij} = t_n\io (-\lambda_{ij}) (u_{n,i}^+)^{\alpha_{ij}+1}(u_{n,j}^+)^{\beta_{ij}} \le \ol d
\]
for some $\ol b,\ol d>0$. So Lemma \ref{su} asserts that $(s^{t_n}_{u_n,i})$ is bounded and bounded away from $0$ for each $i$. Therefore, after passing to a subsequence, $s^{t_n}_{u_n,i}\to s_i>0$, $t_n\to t$ and $u_n\rh u$ weakly in $\cH$. By Lemma \ref{lem:K compact}, $K_{t_n}(s_{u_n}^{t_n}u_n) \to K_t(su)$ strongly in $\cH$, and we easily deduce that $P_{k_n}K_{t_n}(s_{u_n}^{t_n}u_n) \to K_t(su)$ strongly in $\cH$. Now we derive from \eqref{eq:critinV} that $s_{u_n}^{t_n}u_n\to su$ strongly in $\cH$ and $su-K_{t}(su)=0$. Therefore, $su\in\cN_t$, $s=s^t_u$ and $S_t(u)=0$. On the other hand, as $u_n\notin\cV_{t_n}$, we have that $\|s^{t_n}_{u_n}u_n\|\geq R$. Hence, $\|s^t_uu\|\geq R$. This is a contradiction.
\end{proof}

\begin{lemma} \label{contraction}
$\Psi^c\cap E_k$ is contractible in itself for each large enough $k$.
\end{lemma}

\begin{proof}
Let $\eta: [0,1]\times \cU\to \cU$ be given by
\[
\eta(\tau,u) := \left(\frac{(1-\tau)u_1+\tau u_{0,1}}{\|(1-\tau)u_1+\tau u_{0,1}\|},\ldots,  \frac{(1-\tau)u_\ell+\tau u_{0,\ell}}{\|(1-\tau)u_\ell+\tau u_{0,\ell}\|}\right),
\]
where $u_0$ is the previously chosen minimizer for $\Psi$ on $\cU$. Note that $\eta$ is well defined and maps into $\cU$ because $u_{0,i}>0$ in $\Omega$ and $u_i^+\ne 0$ for all $i$. Moreover, if $u\in E_k$, then $\eta(\tau,u)\in E_k$ for each $k\geq 2$. So $\eta$ is a deformation of $\cU\cap E_k$ into $u_0$ and, in particular, of $\Psi^c\cap E_k$ into $u_0$ in $\cU\cap E_k$. 

We claim that there exists $\delta_0>0$ such that
\begin{equation*}
\io[((1-\tau)u_i+\tau u_{0,i})^+]^{p+1} \ge \delta_0\quad\text{for all \ }\tau\in[0,1], \ u\in\Psi^c, \ i=1,\ldots,\ell.
\end{equation*}
Otherwise, there would exist $\tau_n\in[0,1]$ and $u_n\in\Psi^c$ such that
\begin{equation} \label{eq:tozero}
(1-\tau_n)\io(u_{n,i}^+)^{p+1}\leq\io[((1-\tau_n)u_{n,i}+\tau_n u_{0,i})^+]^{p+1}\to 0
\end{equation}
(the inequality is satisfied because $u_{0,i}>0$).
From \eqref{eq:psi} we see that there exists $\delta>0$ such that $\io(u_i^+)^{p+1}\ge\delta$ for all $u\in\Psi^c$ and all $i$. Hence, $\tau_n\to 1$. Since $(u_n)$ is bounded in $\cH$, a subsequence of $(u_{n,i})$ converges in $L^{p+1}(\Omega)$. Therefore,
$$\io[((1-\tau_n)u_{n,i}+\tau_n u_{0,i})^+]^{p+1}\to\io u_{0,i}^{p+1}\geq\delta,$$
a contradiction to \eqref{eq:tozero}.

So, for every $\tau\in[0,1], \ u\in\Psi^c, \ i=1,\ldots,\ell,$ we have 
\begin{align*}
\io(\eta_i(\tau,u)^+)^{p+1} &= \io \frac{[((1-\tau)u_i+\tau u_{0,i})^+]^{p+1}}{\|(1-\tau)u_i+\tau u_{0,i}\|^{p+1}} \\
& \ge \io [((1-\tau)u_i+\tau u_{0,i})^+]^{p+1} \ge \delta_0,
\end{align*}
and we deduce from \eqref{eq:psi} that there exists $d>c$ such that
\begin{equation*}
\eta(\tau,u)\in \Psi^d\cap E_k\qquad\text{for all \ }\tau\in[0,1], \ u\in\Psi^c\cap E_k, \ k\geq 2.
\end{equation*}

Next we show that $\Psi|_{\,\cU\cap E_k}$ does not have a critical value in $[c,d]$ for any large enough $k$. 
Indeed, if $u_k\in\Psi_c^d$ is a critical point of $\Psi|_{\,\cU\cap E_k}$, then, according to \eqref{eq:psi'},
\[
\langle S_0(u_k), s^0_{u_k}v\rangle = 0 \quad \text{for all } v\in T_{u_k}(\cU\cap E_k),
\]
i.e., $P_kS_0(u_k)=0$. Since $u_k\in\Psi_c^d\subset \Psi^d\smallsetminus\cV_t$ (see \eqref{vt}), $k<k_d$ according to Lemma \ref{lem:critinV}.

Now Proposition \ref{mainprop}$(iii)$ allows us to use the negative gradient flow of $\Psi|_{\,\cU\cap E_k}$ in the standard way to obtain a retraction $\varrho:\Psi^d\cap E_k\to\Psi^c\cap E_k$; see, e.g., \cite[Theorem I.3.2]{ch}. Then, $\varrho\circ\eta: [0,1]\times (\Psi^c\cap E_k)\to \Psi^c\cap E_k$ is a deformation of $\Psi^c\cap E_k$ into a point.
\end{proof}

The following statement is an immediate consequence of Lemma \ref{contraction} and basic properties of homology (see e.g. \cite[Sections III.4 and III.5]{do}).

\begin{corollary} \label{chi}
Denote the $q$-th singular homology with coefficients in a field $\mathbb{F}$ by $\mathrm{H}_q(\cdot)$. Then $\mathrm{H}_0(\Psi^c\cap E_k) = \mathbb{F}$ and $\mathrm{H}_q(\Psi^c\cap E_k) = 0$ for $q\neq 0$. In particular, the Euler characteristic 
$$\chi(\Psi^c\cap E_k) := \sum_{q\ge 0}(-1)^q\dim_\mathbb{F} \mathrm{H}_q(\Psi^c\cap E_k) = 1$$
for every large enough $k$.
\end{corollary}

For $u_0$ as above, let 
$$\sigma_i:\cS\smallsetminus\{-u_{0,i}\}\to(\r u_{0,i})^\perp=:F_i$$
be the stereographic projection. The product $\sigma=(\sigma_1,\ldots,\sigma_\ell)$ of the stereographic projections is a diffeomorphism. So its derivative at $u$
\begin{equation*} 
\sigma'(u): T_u(\cU)\to F := F_1 \times\ldots\times F_\ell
\end{equation*}
is an isomorphism for every $u\in\cU$. Note that, as $u_{0,i}\in E_{2,i}$, we have that $\sigma_i((\cS\cap E_k)\smallsetminus\{-u_{0,i}\})\subset F_i\cap E_k$ for all $k\geq 2$.
\medskip

\begin{proof}[Proof of Theorem \ref{mainthm}]
Let $\cO:=\sigma(\cV_0)$ with $\cV_0$ as in \eqref{cv}. As $u_0\in\cV_0$ we have that $0\in\cO$, and as $\overline{\cV}_0\subset\cU$ and $-u_0\notin\cU$, $\cO$ is bounded in $F$. Set $\cO_k:=\cO\cap E_k$ and $F_k:=F\cap E_k$. Then $\cO_k$ is a bounded open neighborhood of $0$ in $F_k$, and $\overline{\cO}_k\subset\sigma(\Psi^c\cap E_k)$ for $c$ as in \eqref{vt}.

Fix $k_c\in\n$ as in Lemma \ref{lem:critinV}. Recall that
\[
S_t(u) =s^t_uu- K_t(s^t_uu)\in T_u(\cU)\qquad\forall u\in\cU
\]
(see Proposition \ref{mainprop}$(iv)$). Define $G_{t,k}:\sigma(\cU\cap E_k)\to F_k$ by
\begin{equation} \label{gtk}
G_{t,k}(w) := (\sigma'(\sigma^{-1}(w))\circ P_k\circ S_t\circ\sigma^{-1})(w).
\end{equation}
Note that
$$G_{t,k}(w)=0\Longleftrightarrow P_k(S_t(\sigma^{-1}(w)))=0.$$
So, if $k\geq k_c$, $w\in\sigma(\cU\cap E_k)$ and $G_{t,k}(w)=0$, Lemma \ref{lem:critinV} asserts that $w\in\cO_k$. In particular, $G_{t,k}(w)\neq 0$ for every $w\in\partial \cO_k$. From the homotopy and the excision properties of the Brouwer degree we get that
\begin{equation} \label{eq:deg1}
\deg(G_{1,k},\cO_k,0)=\deg(G_{0,k},\cO_k,0)=\deg(G_{0,k},\sigma(\Psi^c\cap E_k),0).
\end{equation}
On the other hand, using \eqref{eq:psi'} and \eqref{gtk} we get
\begin{align*}
(\Psi\circ\sigma^{-1})'(w)z&=\Psi'(\sigma^{-1}(w))[(\sigma^{-1})'(w)z]\\
&=s_{\sigma^{-1}(w)}^0\langle P_k(S_0(\sigma^{-1}(w))),\,(\sigma^{-1})'(w)z\rangle \\
&=s_{\sigma^{-1}(w)}^0\langle (\sigma^{-1})'(w)(G_{0,k}(w)),\,(\sigma^{-1})'(w)z\rangle \\
&=\frac{4s_{\sigma^{-1}(w)}^0}{(\|w\|^2+1)^2}\langle G_{0,k}(w),z\rangle\qquad\forall w\in\sigma(\cU\cap E_k), \ z\in F_k.
\end{align*}
The last identity is obtained by a simple calculation, see e.g. \cite[Lemma 3.4]{l}. Since $(\Psi\circ\sigma^{-1})|_{\sigma(\cU\cap E_k)}$ is of class $\cC^2$ (see \eqref{eq:psi}) and $-(\Psi\circ\sigma^{-1})'(w)$ points into $(\Psi\circ\sigma^{-1})^c$ for all $w\in (\Psi\circ\sigma^{-1})^c_c$, from \cite[Theorem II.3.3]{ch} and Corollary \ref{chi} we obtain
\begin{align} \label{chi1}
\deg(G_{0,k},\sigma(\Psi^c\cap E_k),0)&=\deg(((\Psi\circ\sigma^{-1})|_{\sigma(\cU\cap E_k)})',\sigma(\Psi^c\cap E_k),0)\\
& = \chi(\sigma(\Psi^c\cap E_k)) = \chi(\Psi^c\cap E_k) = 1. \nonumber
\end{align} 
Combining \eqref{eq:deg1} and \eqref{chi1} gives
$$\deg(G_{1,k},\cO_k,0)=1.$$
Hence, for each $k\geq k_c$ there exists $w_k\in\cO_k$ such that $G_{k,1}(w_k) = 0$. Then $u_k := \sigma^{-1}(w_k)\in\cV_0\cap E_k\subset \Psi^c\cap E_k$ satisfies $P_k(S(u_k))=0$, i.e.,
\begin{equation} \label{solution}
s_{u_k}u_k = P_kK(s_{u_k}u_k).
\end{equation} 
As in the proof of Lemma \ref{lem:critinV} (with $t_n$ replaced by $1$ and $s_{u_n}^{t_n}u_{n}$ by $s_{u_k}u_k$) one shows that $(s_{u_k,i})$ is bounded and bounded away from $0$ for each $i$. So passing to a subsequence, $s_{u_k}\to s$ and $u_k\rh u$ weakly in $\cH$. Taking limits in \eqref{solution} and using Lemma \ref{lem:K compact}, we obtain that $s_{u_k}u_k\to su$ strongly in $\cH$ and $su=K(su)$. Hence, $su\in\cN$, $s=s_u$ and $S(u) = s_uu-K(s_uu) = 0$. So, according to Proposition \ref{mainprop}$(v)$, $s_uu$ is a solution to \eqref{eq:system}.  
\end{proof}

\section{Synchronized solutions} \label{sec:sync}

A solution $u=(u_1,\ldots,u_\ell)$ to \eqref{eq:system} is called \emph{synchronized} if $u_i=t_iv$ and $u_j=t_jv$ for some $i\ne j$, $v\in H^1_0(\Omega)\smallsetminus\{0\}$ and $t_1,t_2>0$. In this section we consider a system of 2 equations:
\begin{equation} \label{eq:2eq}
\begin{cases}
-\Delta u_1 = \mu_1 u_1^p + \lambda_{12}u_1^{\alpha_{12}}u_2^{\beta_{12}}, \\
-\Delta u_2 = \mu_2 u_2^p + \lambda_{21}u_2^{\alpha_{21}}u_1^{\beta_{21}}, \\
u_1,u_2\ge 0 \text{ in } \Omega, \quad
u_1,u_2\in H^1_0(\Omega)\smallsetminus\{0\}.
\end{cases}
\end{equation}
Recall that according to our assumptions $\alpha_{12}+\beta_{12}<p$ and $\alpha_{21}+\beta_{21}<p$.

\begin{theorem} \label{sync}
The system \eqref{eq:2eq} has a synchronized solution if and only if $\alpha_{12}+\beta_{12}=\alpha_{21}+\beta_{21}=:q$ and
\begin{equation} \label{lambda}
\frac{\lambda_{12}}{\lambda_{21}} = \left(\frac{\mu_1}{\mu_2}\right)^{(\alpha_{21}-\beta_{12}-1)/(p-1)}.
\end{equation}
\end{theorem}

\begin{proof}
Inserting $u_1=t_1v$, $u_2=t_2v$ into \eqref{eq:2eq} we obtain
\begin{equation*} \label{t1t2}
\begin{cases}
-t_1\Delta v = \mu_1 t_1^pv^p + \lambda_{12}t_1^{\alpha_{12}}t_2^{\beta_{12}}v^{\alpha_{12}+\beta_{12}} \\
-t_2\Delta v = \mu_2 t_2^pv^p + \lambda_{21}t_2^{\alpha_{21}}t_1^{\beta_{21}}v^{\alpha_{21}+\beta_{21}}.
\end{cases}
\end{equation*}
 Dividing the first equation by $t_1$, the second one by $t_2$ and subtracting gives
 \[
 (\mu_1t_1^{p-1}-\mu_2t_2^{p-1})v^p + (\lambda_{12}t_1^{\alpha_{12}-1}t_2^{\beta_{12}}v^{\alpha_{12}+\beta_{12}} - \lambda_{21}t_2^{\alpha_{21}-1}t_1^{\beta_{21}}v^{\alpha_{21}+\beta_{21}}) = 0.
 \]
 So $\alpha_{12}+\beta_{12}=\alpha_{21}+\beta_{21}=q$,
 \[
 \mu_1t_1^{p-1}-\mu_2t_2^{p-1}=0 \qquad \text{and} \qquad \lambda_{12}t_1^{\alpha_{12}-1}t_2^{\beta_{12}} = \lambda_{21}t_2^{\alpha_{21}-1}t_1^{\beta_{21}}.
 \]
 Inserting the solution
 \[
 t_2 = \left(\frac{\mu_1}{\mu_2}\right)^{1/(p-1)}t_1
 \]
 of the first equation into the second one gives \eqref{lambda}.
 
 We have shown that the conditions in Theorem \ref{sync} are necessary. It remains to show  that they are also sufficient. To this aim observe that, if $\alpha_{12}+\beta_{12}=\alpha_{21}+\beta_{21}=:q$, \ \eqref{lambda} holds true,  and  $w$ satisfies
\begin{equation} \label{eq:sync}
-\Delta w = \mu_1 w^p - aw^q,\qquad w\ge 0,\ w\in H_0^1(\Omega)\smallsetminus\{0\},
\end{equation}
with 
 \[
 a := -\lambda_{12}\left(\frac{\mu_1}{\mu_2}\right)^{\beta_{12}/(p-1)},
 \]
 then $\left(w,\left(\frac{\mu_1}{\mu_2}\right)^{1/(p-1)}w\right)$ solves the system \eqref{eq:2eq}.
 Consider the functional
 \[
 \Phi(w) := \frac12\io|\nabla w|^2 +\frac a{q+1}\io(w^+)^{q+1} - \frac{\mu_1}{p+1}\io(w^+)^{p+1}. 
 \]
By standard arguments (see e.g. \cite{str} or \cite{wi}),  $\Phi$ is of class $\cC^1$ and critical points of $\Phi$ are  solutions to the equation
 \begin{equation} \label{eq:w}
 -\Delta w +a(w^+)^q = \mu_1 (w^+)^p.
 \end{equation}
 We shall complete the proof by showing that $\Phi$ has a nontrivial critical point $w\ge 0$. We use the mountain pass theorem (see e.g. \cite{str} or \cite{wi}). By easy calculations (as e.g. in \cite[Proof of Theorem 1.19]{wi}), $\Phi$ has the mountain pass geometry. Here it is important that $p>2$ and $p>q$.  Next we show that $\Phi$ satisfies the Palais-Smale condition. Let $(w_n)$ be such that $\Phi(w_n)\to c$ and $\Phi'(w_n)\to 0$. Then
 \begin{align*}
 c+1+\|w_n\| & \ge \Phi(w_n)-\frac1{p+1}\Phi'(w_n)w_n \\
 & = \left(\frac12-\frac1{p+1}\right)\io|\nabla w_n|^2 +a\left(\frac1{q+1}-\frac1{p+1}\right)\io(w_n^+)^{q+1}
 \end{align*}
for all $n$ large enough. Hence $(w_n)$ is bounded, so passing to a subsequence, $w_n\to w$ weakly in $H^1_0(\Omega)$, and strongly in $L^p(\Omega)$ and $L^q(\Omega)$.  It follows by a standard argument (see e.g. \cite[Proof of Lemma 1.20]{wi}) that $w_n\to w$ strongly also in $H^1_0(\Omega)$. Finally, multiplying \eqref{eq:w} by $w^-$ gives $\io|\nabla w^-|^2 = 0$, so $w^-=0$. The proof is complete.
\end{proof}

\begin{remark}
\emph{
It is easy to show that if $q=p$, then there are no synchronized solutions for $-\lambda_{ij}$ sufficiently large, as is well known in the variational case, see e.g. \cite[Proposition 3.2]{cs}.}
\end{remark}

\begin{remark}
\emph{
Let $\lambda_{ij,n}<0$, $i\ne j$, and $u_n=(u_{n,1},\ldots,u_{n,\ell})$ be a solution to \eqref{eq:system} with $\lambda_{ij}$ replaced by $\lambda_{ij,n}$. It is easy to see that, if the sequence $(u_n)$ is bounded in $\cH$, the components $u_{n,i}$ \emph{separate spatially} as $\lambda_{ij,n}\to-\infty$. More precisely, after passing to a subsequence, $u_{n,i}\to u_i\neq 0$ weakly in $H_0^1(\Omega)$ and strongly in $L^p(\Omega)$ for each $i$, and $u_{i}(x)\cdot u_{ j}(x)=0$ a.e. in $\Omega$ for $i\neq j$. There is an extensive literature on spatial separation of solutions and limiting profiles, under the assumption that the sequence $(u_n)$ is bounded and under different assumptions on the nonlinearities. See e.g. \cite{ctv, cd , sttz} and the references therein.
}
\end{remark}

Obviously, synchronized solutions to \eqref{eq:system} do not separate spatially. So we cannot expect the sequence $(w_n)$ given by \eqref{eq:sync} to be bounded. Indeed, we have the following

\begin{proposition} \label{unbdd}
Let $(w_n)$ be a sequence of solutions to \eqref{eq:sync} with $a=a_n$. If $a_n\to\infty$, then $(w_n)$ is unbounded in $H_0^1(\Omega)$. 
\end{proposition}

\begin{proof}
Suppose $(w_n)$ is bounded. Then, passing to a subsequence, $w_n\to w$ weakly in $H^1_0(\Omega)$, strongly in $L^p(\Omega)$ and in $L^q(\Omega)$. Since
\[
\io|\nabla w_n|^2+a_n\io w_n^q = \mu_1\io w_n^p,
\]
we have that $w_n\to 0$ in $L^q(\Omega)$. So $w=0$ and therefore $w_n\to 0$ strongly in $H^1_0(\Omega)$. This is a contradiction because by the Sobolev inequality,
\[
\io|\nabla w_n|^2 \le \mu_1\io w_n^p \le C\left(\io|\nabla w_n|^2\right)^{p/2}
\]
for some constant $C$, so $\|w_n\|$ is bounded away from $0$.
\end{proof}

It is well known that, when the system \eqref{eq:system} is variational, least energy solutions are bounded in $\cH$, independently of $\lambda_{ij}$. We close this section with the following open question.

\begin{problem} \label{pb}
Given $\lambda_{ij,n}\to-\infty$ for $i\ne j$, does the  system \eqref{eq:system} with $\lambda_{ij}$ replaced by $\lambda_{ij,n}$ have a solution $u_n$ such that the sequence $(u_{n,i})$ is bounded in $H_0^1(\Omega)$ for all $i$?
\end{problem}

\appendix
\section{Appendix} \label{sec:appendix}

In this appendix we prove Lemma \ref{Linfty}. We employ some arguments which may be found in \cite{dFY, gs}. First we note that by standard regularity results the solutions $u_i$ of \eqref{eq:system3} are in $\cC^2(\Omega)\cap \cC(\overline\Omega)$. 

Suppose there exists a sequence of solutions $(u_n)$ with $|u_n|_\infty\to\infty$. Passing to a subsequence, we may assume $|u_{n,i}|_\infty\to\infty$ and  $|u_{n,i}|_\infty\ge|u_{n,j}|_\infty$ for some $i$ and all $j$. There exists $x_n\in\Omega$ such that 
\[
\max_{x\in\Omega} u_{n,i}(x) = u_{n,i}(x_n).
\]
Let $\beta := \frac2{p-1}$ and choose $\vr_n$ so that
\[
\vr_n^\beta\,|u_{n,i}|_\infty = 1.
\]
Then $\vr_n\to 0$ and passing to a subsequence, $x_n\to x_0\in\overline{\Omega}$. Let 
\[
\Omega_n := \{y\in\rn: \vr_ny+x_n\in\Omega\}
\]
and
\begin{equation} \label{defvni}
v_{n,j}(y) := \vr_n^\beta u_{n,j}(\vr_ny+x_n), \quad j=1,\ldots, \ell.
\end{equation}
Then,
\begin{equation} \label{vni}
0\le v_{n,i}\le1, \quad v_{n,i}(0)=1 \quad \text{and } v_{n,j}|_{\partial\Omega_n} = 0  \text{ for all } j.
\end{equation}
Passing to a subsequence, there are two possible cases and we shall complete the proof by ruling out both of them. Denote the distance from $x$ to a set $A$ by $d(x,A)$.

\medskip

\noindent\emph{Case 1.} $\frac{d(x_n,\partial\Omega)}{\vr_n}\to\infty$. \\
Since $\vr_ny+x_n\in\Omega$ if $|y| < \frac{d(x_n,\partial\Omega)}{\vr_n}$, for each $R>0$ there exists $n_0$ such that $B_R(0)\subset \Omega_n$ whenever $n\ge n_0$. For $y\in B_R(0)$ and $n\ge n_0$ we have
\begin{align*}
-\Delta_y v_{n,i} & = \vr_n^{\beta+2}\Delta_x u_{n,i} = \vr_n^{\beta+2}\Big(\mu_iu_{n,i}^p+\sum\limits_{j\ne i}\lambda_{ij}u_{n,i}^{\alpha_{ij}}u_{n,j}^{\beta_{ij}}\Big) \\
& = \vr_n^{\beta+2-\beta p}\mu_i v_{n,i}^p + \sum\limits_{j\ne i} \vr_n^{\beta+2-\beta(\alpha_{ij}+\beta_{ij})}\lambda_{ij}v_{n,i}^{\alpha_{ij}}v_{n,j}^{\beta_{ij}}.
\end{align*}
Since $\beta+2-\beta p = 0$ and $\gamma_{ij} := \beta+2-\beta(\alpha_{ij}+\beta_{ij}) > 0$, we can re-write this identity as
\[
-\Delta v_{n,i} = \mu_i v_{n,i}^p + \sum\limits_{j\ne i} \vr_n^{\gamma_{ij}}\lambda_{ij}v_{n,i}^{\alpha_{ij}}v_{n,j}^{\beta_{ij}}.
\]
By elliptic estimates, $(v_{n,i})$ is bounded in $W^{2,q}(B_R(0))$ for some $q>N$. So passing to a subsequence, $v_{n,i} \to v_i$ weakly in $W^{2,q}(B_R(0))$ and strongly in $\cC^1(B_R(0))$. Since $\vr_n^{\gamma_{ij}}\to 0$, $v_i$ is a nonnegative solution to the equation  
\[
-\Delta v = \mu_iv^p
\]
in $B_R(0)$. Let now $R_m\to\infty$. Then for each $m$ we get a solution $v_{i,m}$ of the above equation in $B_{R_m}(0)$. Passing to subsequences and applying the diagonal procedure, we see that $v_{i,m}\to w$, weakly in $W^{2,q}_{loc}(\rn)$ and strongly in $\cC^1_{loc}(\rn)$. So $-\Delta w = \mu_iw^p$ in $\rn$, $w\ge 0$, $w(0)=1$ according to \eqref{vni}, and $w\in \cC^2(\rn)$ by Schauder estimates. Replacing $w$ with $cw$ for a suitable $c>0$ we may assume $\mu_i=1$. Hence it follows from \cite[Theorem 1.2]{gs} that $w=0$ which rules out Case 1.

\medskip

\noindent\emph{Case 2.} $\frac{d(x_n,\partial\Omega)}{\vr_n}\to d\in[0,\infty)$. \\
It is clear that $x_0\in\partial\Omega$ and we may assume without loss of generality that $x_0=0$ and $\nu = (0,\ldots,0,1)$ is the unit outer normal to $\partial\Omega$ at $x_0$. Let
\[
\hn := \{y\in\rn: y_N<d\} \quad \text{where } y=(y_1,\ldots,y_N).
\]
We shall need the following result.

\begin{lemma} \label{lem}
\begin{itemize}
\item[$(i)$]Let $A\subset \hn$ be compact. Then there exists $n_0$ such that $\vr_ny+x_n\in\Omega$ for all $n\ge n_0$ and $y\in A$.
\item[$(ii)$]Let $A\subset \rn\smallsetminus\ol\hn$ be compact. Then there exists $n_0$ such that $\vr_ny+x_n\notin\Omega$ for all $n\ge n_0$ and $y\in A$.
\end{itemize}
\end{lemma}

\begin{proof}
$(i):$ Since $A$ is compact, there exists $\eps>0$ such that $y_N<d-2\eps$ for all $y\in A$. For each $n$ there exists $\wh x_n\in\partial\Omega$ which is closest to $x_n$, i.e., $d(x_n,\partial\Omega) = |x_n-\wh x_n|$. As $\partial\Omega$ is tangent to the hyperplane $x_N=0$ at $0$,
\[
\frac{|x_n-\wh x_n|}{\vr_n} = \frac{\wh x_{n,N}-x_{n,N}}{\vr_n} + o(1).
\]
Therefore,
\[
\frac{x_{n,N}-\wh x_{n,N}}{\vr_n} < -d+\eps \qquad \text{and} \qquad y_{N} + \frac{x_{n,N}-\wh x_{n,N}}{\vr_n} < -\eps
\] 
for all $y\in A$ if $n$ is large enough. There exists $C>0$ such that 
\[
\left|y+\frac{x_n-\wh x_n}{\vr_n}\right| \le C \quad \text{for all } y\in A.
\]
Using this, we see that there is $n_0$ such that, if $n\ge n_0$ and $y\in A$, then $\vr_ny+x_n-\wh x_n\in\Omega$ and, as $\wh x_n\in\partial\Omega$ and $\partial\Omega$ is tangent to the hyperplane $x_N=0$ at $0$, \ $\vr_ny+x_n = \vr_ny+(x_n-\wh x_n) + \wh x_n\in \Omega$. 

$(ii):$ This time $y_N>d+2\eps$ for $y\in A$, 
\[
\frac{x_{n,N}-\wh x_{n,N}}{\vr_n} > -d-\eps \quad \text{and} \quad y_{n,N} + \frac{x_{n,N}-\wh x_{n,N}}{\vr_n} > \eps
\]
if $n$ is sufficiently large, and the conclusion follows by a similar argument as above. 
\end{proof}

Now we can continue with Case 2. Let $\omega_R := B_R(0)\cap\{y\in\rn: y_N<d-1/R\}$. Then $\overline{\omega}_R\subset \Omega_n$ for all $n\ge n_0$ by Lemma \ref{lem}$(i)$. Let $v_{n,i}$ be given by \eqref{defvni} and using \eqref{vni} extend it by 0 outside $\Omega_n$.  According to Lemma \ref{lem}$(ii)$, if $A\subset \rn\smallsetminus\ol\hn$ is compact, then $\vr_ny+x_n\notin\Omega$ for all $y\in A$ and $n$ large enough. So 
\begin{equation} \label{lim}
\lim_{n\to\infty}v_{n,i}(x)=0 \quad \text{for all } x\notin \hn.
\end{equation} 
We can repeat the argument of Case 1 which now gives a nonegative solution to the equation $-\Delta w = \mu_iw^p$ in $\hn$ such that $w(0)=1$. By \eqref{lim}, $w=0$ on $\partial\Omega$. As before, $w\in \cC^2(\hn)$, and since the extended functions $v_{n,i}$ are continuous in $\rn$, $w\in \cC^0(\overline{\hn})$. So $w=0$ according to \cite[Theorem 1.3]{gs}, a contradiction. Hence also Case 2 is ruled out.

\bigskip

\begin{flushleft}
\textbf{Mónica Clapp}\\
Instituto de Matemáticas\\
Universidad Nacional Autónoma de México\\
Circuito Exterior, Ciudad Universitaria\\
04510 Coyoacán, Ciudad de México, Mexico\\
\texttt{monica.clapp@im.unam.mx} 
\medskip

\textbf{Andrzej Szulkin}\\
Department of Mathematics\\
Stockholm University\\
106 91 Stockholm, Sweden\\
\texttt{andrzejs@math.su.se} 
\end{flushleft}

\end{document}